\documentclass[11pt,reqno]{amsart}
\usepackage[a4paper, hmargin={2.7cm,2.7cm},vmargin={3.3cm,3.3cm}]{geometry}
\usepackage[english]{babel}
\usepackage{color}
\usepackage[noadjust]{cite}
\usepackage{amsmath}
\usepackage{amssymb}
\usepackage{amsfonts}
\usepackage{amsthm}
\usepackage{enumerate}
\usepackage{amsthm}
\usepackage{bbm}
\usepackage{dsfont}

\usepackage[textwidth=20mm]{todonotes}
\usepackage{cleveref}
\usepackage{commath}

\usepackage{etoolbox}

\makeatletter
\patchcmd{\ttlh@hang}{\parindent\z@}{\parindent\z@\leavevmode}{}{}
\patchcmd{\ttlh@hang}{\noindent}{}{}{}
\makeatother

\newcommand\numberthis{\addtocounter{equation}{1}\tag{\theequation}}

\newtheorem{theorem}{Theorem}[section]
\newtheorem{lemma}[theorem]{Lemma}
\newtheorem{proposition}[theorem]{Proposition}

\theoremstyle{definition}

\newtheorem{example}[theorem]{Example}
\newtheorem{question}[theorem]{Question}

\theoremstyle{remark}

\numberwithin{equation}{section}

\def\XXint#1#2#3{{\setbox0=\hbox{$#1{#2#3}{\int}$ }
\vcenter{\hbox{$#2#3$ }}\kern-.6\wd}}

\newcommand{\Z}{\mathbb{Z}}
\newcommand{\N}{\mathbb{N}}
\newcommand{\R}{\mathbb{R}}
\newcommand{\C}{\mathbb{C}}

\DeclareMathOperator{\spn}{span}

\DeclareMathOperator{\supp}{supp}
\DeclareMathOperator{\interior}{int}

\DeclareMathOperator{\tr}{tr}

\newcommand{\FS}{\mathbb{C} \Gamma}

\newcommand{\Hpi}{\mathcal{H}_{\pi}}

\def\pker#1{P_{#1}}

\title[Linear independence of coherent systems associated to discrete subgroups]{Linear independence of   coherent systems \\ associated to discrete subgroups}

\makeatletter
\@namedef{subjclassname@2020}{\textup{2020} Mathematics Subject Classification}
\makeatother



\author{Ulrik Enstad}
\address{Department of Mathematics,
University of Oslo,
Moltke Moes vei 35,
0851 Oslo.}
\email{ubenstad@math.uio.no}

\author{Jordy Timo van Velthoven}
\address{Faculty of Mathematics,
University of Vienna, 
Oskar-Morgenstern-Platz 1,
1090 Vienna, Austria}
\email{jordy-timo.van-velthoven@univie.ac.at}


\begin{document}

\maketitle

\begin{abstract}
This note considers the finite linear independence of coherent systems associated to discrete subgroups. We show by simple arguments that such coherent systems of amenable groups are linearly independent whenever the associated twisted group ring does not contain any nontrivial zero divisors. We verify the latter for discrete subgroups in nilpotent Lie groups.
For the particular case of time-frequency translates of Euclidean space, our approach provides a simple and self-contained proof of the Heil--Ramanathan--Topiwala (HRT) conjecture for subsets of arbitrary discrete subgroups.
\end{abstract}

\section{Introduction}
Given a point $(x,\xi) \in \R^d \times \R^d = \R^{2d}$,  the associated time-frequency translation $\pi(x, \xi)$ is the unitary operator on $L^2(\R^d)$ given by
\begin{align} \label{eq:TF}
\pi(x,\xi) g(t) = e^{2 \pi i \xi \cdot t} g(t-x), \quad  t \in \mathbb{R}^d. 
\end{align}
The famous Heil--Ramanathan--Topiwala (HRT) conjecture \cite[Conjecture]{heil1996linear} states that for any nonzero $g \in L^2 (\mathbb{R}^d)$ and any finite subset $\Lambda \subseteq \mathbb{R}^{2d}$, the  system 
\begin{align} \label{eq:gabor}
\pi (\Lambda) g = \{ \pi(\lambda) g : \lambda \in \Lambda \}
\end{align}
is linearly independent in $L^2(\R^d)$. See the surveys \cite{heil2006linear, heil2015hrt} for an overview of the background, motivation and many partial results on the conjecture obtained so far.

One of the most important partial results on the HRT conjecture is the following theorem proved in \cite{linnell1999neumann}.

\begin{theorem} \label{thm:linnell}
Let $\Lambda \subseteq \Gamma$ be a finite subset of a discrete subgroup $\Gamma \leq \mathbb{R}^{2d}$. 
For any nonzero $g \in L^2 (\mathbb{R}^d)$, the system $\pi(\Lambda) g$ is linearly independent. 
\end{theorem}

The proof of \Cref{thm:linnell} given in \cite{linnell1999neumann} uses group von Neumann algebras and analytic versions of the zero divisor conjecture \cite{linnell1991zero}. 
For dimension $d = 1$, alternative proofs of \Cref{thm:linnell} with more analytic arguments have been given in \cite{demeter2013on, antezana2020linear, bownik2010linear}. See also \cite{oussa2019hrt} for an approach towards \Cref{thm:linnell} through analysis on the Heisenberg group and representation theory.

The first aim of the present note is to show that the linear dependence of a system as in \Cref{thm:linnell} implies the existence of nontrivial zero divisors in a twisted group ring of $\mathbb{Z}^{n}$. Since no such zero divisors exist (see \Cref{ex:integers_zero_divisor}), this provides a simple and self-contained proof of \Cref{thm:linnell}. The relation (or lack thereof) between the HRT conjecture and the zero divisor conjecture for the ordinary group ring of the Heisenberg group is the topic of the survey \cite{heil2015hrt}. 
Our second aim is to obtain extensions of \Cref{thm:linnell} that cover new classes of groups and  representations. We expand on both aspects in the next subsections.

\subsection{Proof strategy} \label{sec:proof}
The approach towards Theorem \ref{thm:linnell} used in this note is based on a relation between linear independence of orbits of discrete subgroups \eqref{eq:gabor} and the nonexistence of zero divisors for twisted convolution. To be more explicit, let $\sigma$ be the 2-cocycle coming from the composition rule for time-frequency translates
\[ \pi(z)\pi(z') = \sigma(z,z')\pi(z+z') , \qquad z,z' \in \R^{2d},  \]
that is,
\begin{equation}
    \sigma(z,z') = e^{-2\pi i x \cdot \xi'}, \quad z = (x,\xi), z' = (x',\xi') . \label{eq:time-freq-cocycle}
\end{equation}
For a discrete subgroup $\Gamma \leq \mathbb{R}^{2d}$, the $\sigma$-twisted convolution of two finitely supported, complex-valued sequences $a,b$ on $\Gamma$ is given by
\[
(a \ast_{\sigma} b)(\gamma') = \sum_{\gamma \in \Gamma} a(\gamma) b(\gamma' - \gamma) \sigma(\gamma, \gamma' - \gamma), \quad \gamma' \in \Gamma.
\]
 The set of finitely supported sequences on $\Gamma$ equipped with $\sigma$-twisted convolution forms an algebra $\C(\Gamma,\sigma)$, which is often called a \emph{twisted group ring} of $\Gamma$. An element $a \in \C(\Gamma,\sigma)$ is said to be a \emph{zero divisor} if there exists nonzero $b \in \C(\Gamma,\sigma)$ such that $a \ast_{\sigma} b =  0$.

Our proof method for \Cref{thm:linnell} consists of showing that the lack of nontrivial zero divisors in the twisted group ring $\C(\Gamma,\sigma)$ implies the linear independence of the system $\pi(\Gamma)g$ for any nonzero $g \in L^2(\R^d)$. Arguing by contraposition, we show that the existence of a nonzero function with linearly dependent time-frequency translates would imply the existence of a nonzero function in $L^2(\R^{2d})$ with a linearly dependent set of $\sigma$-twisted translates (\Cref{lem:lineardependG}). This further implies the existence of a nonzero sequence in $\ell^2(\Gamma)$ with linearly dependent twisted translates (\Cref{lem:reductionGamma}), which ultimately implies the existence of a nontrivial zero divisor in $\C(\Gamma,\sigma)$ (\Cref{prop:l2zerodivisor}). We provide a simple direct proof for the nonexistence of zero divisors in $\C(\Gamma,\sigma)$ in \Cref{ex:integers_zero_divisor}.

For ordinary (nontwisted) convolution, several methods used in this note can already be found in the literature as pointed out throughout the text. In particular, the relation between linear independence of translates and zero divisors for convolution appears already in \cite{linnell2017linear}, and the reduction from linear independence of (nontwisted) translates in $\ell^2(\Gamma)$ to the existence of nontrivial zero divisors in the (nontwisted) group ring $\C\Gamma$ is contained in \cite{elek2003analytic}. These results were used in \cite[Proposition 6.3]{linnell2017linear} to prove a special case of Theorem \ref{thm:linnell} for discrete subgroups of $\mathbb{R}^{2d}$ generated by finitely many points $(a_1, b_1), \ldots, (a_n,b_n) \in \mathbb{R}^{2d}$ such that $a_k \cdot b_{k'} \in \mathbb{Q}$ for $1 \leq k,k' \leq n$. However, it appears that the general version of \Cref{thm:linnell} for arbitrary discrete subgroups cannot be deduced from \cite{elek2003analytic, linnell2017linear}.

Our main contribution is to show that by using twisted convolution (instead of ordinary convolution) a streamlined proof of the general version of \Cref{thm:linnell} can be obtained. It is expected that the proof presented here is accessible to all interested readers.

\subsection{Extensions}
 The overall proof approach outlined in Section \ref{sec:proof} works naturally in the general setting of amenable locally compact groups. This allows us to also obtain extensions of \Cref{thm:linnell} to more general groups and projective representations. Such extensions are of interest in view of comments made on \cite[p. 2790]{heil1996linear}, where it is written that understanding the linear independence problem more generally would be of great interest.

One natural setting to which \Cref{thm:linnell} extends is the class of square-integrable projective representations of nilpotent Lie groups. The time-frequency translations given in \Cref{eq:TF} provide the easiest example of such a projective representation, and generally any such representation $\pi$ can be realized to act in a Hilbert space $\Hpi = L^2 (\mathbb{R}^n)$ by means of generalized modulations and translations, cf. \cite{corwin1990representations}. This class of representations has recently been used for various forms of generalized time-frequency analysis, see, e.g., \cite{grochenig2018orthonormal, groechenig2021new, gabardo2021on, oussa2022orthonormal, fischer2018heisenberg, bedos2022smooth, grochenig2020balian}.

Our main result for nilpotent Lie groups is the following theorem:

 \begin{theorem} \label{thm:intro2}
     Let $G$ be a connected, simply connected, nilpotent Lie group and let $\pi$ be an irreducible, square-integrable, projective unitary representation on a Hilbert space $\Hpi$. For any finite subset $\Lambda \subseteq \Gamma$ of a discrete subgroup $\Gamma \leq G$ and any nonzero $g \in \Hpi$,
     the coherent system
     \[
      \pi (\Lambda) g = \{ \pi(\lambda) g : \lambda \in \Lambda \}
     \]
 is linearly independent. 
 \end{theorem}

 \Cref{thm:intro2} might fail without the assumption that the projective representation $\pi$ is square-integrable. Indeed, for a general irreducible projective representation $\pi$, its projective kernel
\[
\pker{\pi} := \{ x \in G : \pi (x) \in \mathbb{T} I_{\Hpi} \}
\]
might be nontrivial, which clearly yields a linear dependent system in its orbit. However, the assumption that $\pi$ is square-integrable yields that $\pker{\pi}$ is a compact subgroup of the connected, simply connected nilpotent group $G$ (see, e.g., \cite[Lemma 2.3]{enstad2022sufficient}), and thus it must be trivial. 

 We mention that the linear independence of subsystems of orbits of square-integrable unitary representations of locally compact groups has been studied earlier in \cite{linnell2017linear}. However, the paper \cite{linnell2017linear} treats only genuine (nonprojective)  representations that are square-integrable in the strict sense, which do not exist for (simply connected) nilpotent Lie groups. To treat projective representations, the use of twisted convolutions as explained in \Cref{sec:proof} appears to be essential.

It is a natural question to what generality the statements of \Cref{thm:linnell} and \Cref{thm:intro2} can be extended. Our most general statement, \Cref{thm:intro3}, from which \Cref{thm:intro2} (and \Cref{thm:linnell}) are obtained, states that for a $\sigma$-projective unitary representation $(\pi,\Hpi)$ with admissible vectors,  the coherent system $\pi(\Gamma)g$ ($\Gamma \subseteq G$ discrete subgroup and $g \in \Hpi$ nonzero) is linearly independent whenever the twisted group ring $\C(\Gamma,\sigma)$ contains no nontrivial zero divisors. Thus, the problem of linear independence is reduced to the algebraic problem of zero divisors in twisted group rings. For the trivial 2-cocycle $\sigma \equiv 1$, the famous zero divisor conjecture predicts that the ordinary (nontwisted) group ring $\C\Gamma$ contains no nontrivial zero divisors whenever $\Gamma$ is torsion-free. More generally, one can ask whether $\C(\Gamma,\sigma)$ contains no nontrivial zero divisors for a torsion-free group $\Gamma$, where $\sigma$ is an arbitrary 2-cocycle on $\Gamma$ (see \Cref{qu:twisted_zero_divisor}). A positive answer to this question would imply, via \Cref{thm:intro3}, the linear independence of $\pi(\Gamma)g$ for all discrete subgroups $\Gamma$ of torsion-free, amenable groups $G$.

\subsection*{Notation} The support of a sequence $a : \Gamma \to \mathbb{C}$ is denoted by $\supp(a) = \{\gamma \in \Gamma : a(\gamma) \neq 0 \}$. The space of all complex-valued sequences on $\Gamma$ with finite support is denoted by $\FS$. For $\gamma \in \Gamma$, the sequence $\delta_{\gamma} : \Gamma \to \mathbb{C}$ is defined by $\delta_{\gamma} (\gamma) = 1$ and $\delta_{\gamma} (\gamma') = 0$ for $\gamma \neq \gamma'$. Denoting by $\mathcal{B}(\mathcal{H})$ the space of bounded linear operators on a Hilbert space $\mathcal{H}$, the commutant of a subset $M \subseteq \mathcal{B}(\mathcal{H})$ is the set $M' = \{ T \in \mathcal{B}(\mathcal{H}) : TS = ST \; \text{for all} \; S \in \mathcal{B}(\mathcal{H})\}$.

\section{Linear dependence of twisted left translates} \label{sec:linear}
The purpose of this section is to reduce the problem of linear dependence of an orbit of a square-integrable projective representation into determining the linear dependence of twisted translates of finite sequences on a discrete subgroup. The results in this section are inspired by corresponding results in \cite{linnell2017linear, elek2003analytic}.

Throughout, let $G$ be a second-countable locally compact group with identity element $e$. A \emph{2-cocycle} on $G$ is a measurable function $\sigma \colon G \times G \to \mathbb{T}$ that satisfies the following properties:
\begin{align} \label{eq:cocycle}
    \sigma(x,y)\sigma(xy,z) &= \sigma(x,yz)\sigma(y,z), \qquad x,y,z \in G, \\
    \sigma(e,e) &= 1 .
\end{align}
A \emph{$\sigma$-projective unitary representation} $(\pi, \Hpi)$ of $G$ on a Hilbert space $\Hpi$ is a 
measurable map $\pi : G \to \mathcal{U}(\Hpi)$ satisfying
\[
\pi(x) \pi(y)  = \sigma(x,y) \pi(xy) , \quad \text{for all $x,y \in G$},
\]
for some function $\sigma : G \times G \to \mathbb{C}$, which is necessarily a $2$-cocycle. 

The projective representation $\pi$ will always be assumed to have an \emph{admissible vector}, that is, a vector $h \in \Hpi$ such that the coefficient transform $V_h : \Hpi \to L^{\infty} (G)$ given by
\[
V_h f(x) = \langle f, \pi(x) h \rangle , \qquad f \in \Hpi ,
\]
defines an isometry into $L^2(G)$. If $\pi$ is irreducible, i.e., if $\{0\}$ and $\Hpi$ are the only closed subspaces $\mathcal{K}$ of $\Hpi$ such that $\pi(x) \mathcal{K} \subseteq \mathcal{K}$ for any $x \in G$, then any nonzero vector $h \in \Hpi$ satisfying $V_h h \in L^2 (G)$ is (a multiple of) an admissible vector by  the orthogonality relations, see, e.g., \cite{duflo1976on, carey1976square}. 
See also \cite{fuehr2005abstract, currey2012admissibility} 
for various classes of possibly reducible representations.

The significance of a $\sigma$-projective representation $(\pi, \Hpi)$ admitting an admissible vector is that it is unitarily equivalent to a subrepresentation of the \emph{twisted left-regular representation} $(\lambda_G^{\sigma}, L^2 (G))$, which is defined by the action
\begin{align} \label{eq:twistedtranslation}
(\lambda_G^{\sigma} (x) F)(y) = \sigma(x, x^{-1} y) F(x^{-1} y), \quad x,y \in G,
\end{align}
for $F \in L^2 (G)$.  Indeed, if $h \in \Hpi$ is admissible, then 
$
V_h : \Hpi \to L^2 (G)$
unitarily intertwines $\pi$ and $\lambda_G^{\sigma}$, in the sense that
\begin{align} \label{eq:covariance}
V_h (\pi(x) f)(y) = \sigma(x, x^{-1} y) V_h f(x^{-1} y) = (\lambda_G^{\sigma} (x) V_h f) (y), \quad x,y \in G,
\end{align}
for any $f \in \Hpi$. 

\subsection{Linear dependence of translates in $L^2 (G)$}
The covariance relation \eqref{eq:covariance} immediately yields the following result.

\begin{lemma} \label{lem:lineardependG}
Let $\Gamma \subseteq G$. If there exist a nonzero $g \in \Hpi$ such that $\pi(\Gamma) g$ is linearly dependent,  then there exists a nonzero $F \in L^2 (G)$ such that $\lambda_G^{\sigma} (\Gamma)F$ is linearly dependent. 
\end{lemma}
\begin{proof}
Let $g \in \Hpi$ be nonzero and suppose that there exist constants $\alpha_1, \alpha_2, ..., \alpha_n \in \mathbb{C}$, not all zero, and points $\gamma_1, \gamma_2, ..., \gamma_n \in \Gamma$ such that $\sum_{k = 1}^n \alpha_k \pi(\gamma_k) g = 0$. Let $h$ be an admissible vector. Then $F :=V_h g$ is a nonzero element of $L^2(G)$, and hence, by Equation  \eqref{eq:covariance},
\[
\sum_{k = 1}^n \alpha_k \lambda_G^{\sigma} (\gamma_k) F = V_h \bigg( \sum_{k = 1}^n \alpha_k \pi (\gamma_k) g \bigg) = 0 ,
\]
as desired.
\end{proof}

\subsection{Linear dependence of translates in $\ell^2 (\Gamma)$}
The aim of this subsection is to reduce the problem of linear dependence of twisted translates on $L^2(G)$ along a discrete subgroup $\Gamma$ to twisted translates on $\ell^2(\Gamma)$. For this, we denote by $\lambda_\Gamma^\sigma$ the $\sigma$-projective left regular representation of $\Gamma$ on $\ell^2(\Gamma)$ (cf. \Cref{eq:twistedtranslation}), where the 2-cocycle $\sigma$ on $G$ is restricted to $\Gamma$.

\begin{lemma} \label{lem:reductionGamma}
Let $\Gamma \leq G$ be a discrete subgroup.
If there exists a nonzero $F \in L^2 (G)$ such that $\lambda_G^{\sigma} (\Gamma) F$ is linearly dependent, then there exists a nonzero $c \in \ell^2 (\Gamma)$ such that $\lambda_{\Gamma}^{\sigma} (\Gamma) c$ is linearly dependent. 
\end{lemma}
\begin{proof}
By  \cite[Proposition B.2.4]{bekka2008kazhdan} , there exists a fundamental domain for $\Gamma$ in $G$, that is, a Borel set $\Omega \subseteq G$ such that $G$ is the disjoint union of the sets $\gamma \Omega$ for $\gamma \in \Gamma$. Consequently any $F \in L^2 (G)$ can be represented by the norm convergent series 
\begin{align} \label{eq:sum_fundom}
F = \sum_{\gamma \in \Gamma} F \cdot \mathds{1}_{\gamma \Omega} .
\end{align}
Let $\mathcal{H}_{\Omega}$ denote the closed subspace of $L^2(G)$ consisting of functions whose essential support is contained in $\Omega$, and choose an orthonormal basis $(e_i)_{i \in \N}$ for $\mathcal{H}_{\Omega}$. Then, for each $\gamma \in \Gamma$, the set $(\lambda_G^\sigma(\gamma)e_i)_{i \in \N}$ is an orthonormal basis for $\lambda_G^\sigma(\gamma)\mathcal{H}_\Omega$, which consists exactly of the functions in $L^2(G)$ that are essentially supported on $\gamma \Omega$. Since $\gamma \Omega$ and $\gamma' \Omega$ are disjoint for $\gamma \neq \gamma'$, it follows that $\lambda_G^\sigma(\gamma)\mathcal{H}_\Omega$ and $\lambda_G^\sigma(\gamma')\mathcal{H}_\Omega$ have trivial intersection when $\gamma \neq \gamma'$. In combination with \Cref{eq:sum_fundom}, this shows that $(\lambda_G^\sigma(\gamma)e_i)_{\gamma \in \Gamma, i \in \N }$ is an orthonormal basis for $L^2(G)$. 

For fixed $i \in \mathbb{N}$, define
\[ \mathcal{K}_i = \overline{ \spn \{ \lambda_G^\sigma(\gamma)e_i : \gamma \in \Gamma \} } . \]
Then $L^2(G) = \bigoplus_{i \in \N} \mathcal{K}_i$. Furthermore, each $\mathcal{K}_i$ is invariant under the operators $\lambda_G^\sigma(\gamma)$ for $\gamma \in \Gamma$, so the orthogonal projection $P_i \in \mathcal{B}(L^2(G))$ onto $\mathcal{K}_i$ commutes with these operators. Let also $T_i \colon \ell^2(\Gamma) \to \mathcal{K}_i$ be the surjective linear isometry given by sending $\lambda_\Gamma^\sigma(\gamma)\delta_e$ to $\lambda_G^\sigma(\gamma)e_i$ for $\gamma \in \Gamma$. Evidently, $T_i\lambda_\Gamma^\sigma(\gamma)=\lambda_G^\sigma(\gamma)T_i$ for $\gamma \in \Gamma$ and $i \in \N$.

For proving the claim, suppose that $F \in L^2 (G)$ is nonzero and that there exist scalars $\alpha_1, ..., \alpha_n \in \mathbb{C}$ not all equal to zero and $\gamma_1, ..., \gamma_n \in \Gamma$ such that $\sum_{k = 1}^n \alpha_k \lambda_{G}^{\sigma} (\gamma_k) F = 0$. Since $F \neq 0$, there exists $i' \in \mathbb{N}$ such that $P_{i'}F \neq 0$. Set $c = T_{i'}^{-1}P_{i'}F$. Then
\[
\sum_{k = 1}^n \alpha_k \lambda_{\Gamma}^{\sigma} (\gamma_k) c = T_{i'}^{-1} P_{i'} \bigg( \sum_{k = 1}^n \alpha_k \lambda_G^{\sigma} (\gamma_k) F \bigg) = 0,
\]
as required.
\end{proof}

The above proof employs the standard technique of decomposing $\lambda_G^\sigma|_{\Gamma}$ into a countable direct sum with $\lambda_\Gamma^\sigma$ as summands. In \cite[Proposition 5.1]{linnell2017linear}, the same technique was used for ordinary (nonprojective) representations.

\subsection{Linear dependence of translates in $\mathbb{C} \Gamma$} 
The purpose of this section is to reduce the linear dependence of twisted translates in $\ell^2(\Gamma)$ even further to finitely supported sequences on $\Gamma$ (see \Cref{prop:l2zerodivisor}) under the assumption of amenability of the discrete group. This is the key step in the proofs of our main theorems and is an adaption of \cite[Theorem]{elek2003analytic} to the projective setting.

In addition to the left regular representation, we will also use the $\overline{\sigma}$-twisted right regular representation $(\rho_\Gamma^{\sigma},\ell^2(\Gamma))$ of $\Gamma$, given by
\[
(\rho_\Gamma^{\sigma} (\gamma) c)({\gamma'}) = \overline{\sigma(\gamma', \gamma)} c(\gamma' \gamma), \quad c \in \ell^2(\Gamma), \gamma,\gamma' \in \Gamma.
\]
A basic fact is that $\rho_\Gamma^{\sigma}$ commutes with $\lambda_\Gamma^\sigma$, that is, $\lambda_\Gamma^\sigma(\gamma) \rho_\Gamma^{\sigma}(\gamma') = \rho_\Gamma^{\sigma}(\gamma') \lambda_\Gamma^\sigma(\gamma)$ for all $\gamma,\gamma'\in \Gamma$. In fact, the commutant $\rho_\Gamma^{\sigma}(\Gamma)'$ is equal to the closure of the span of $\lambda_\Gamma^\sigma(\Gamma)$ in the weak operator topology on $\mathcal{B}(\ell^2(\Gamma))$, cf. \cite[Theorem 1]{kleppner1962structure}. On $\rho_\Gamma^{\sigma}(\Gamma)'$, we define the map $\tau \colon \rho_\Gamma^{\sigma}(\Gamma)' \to \C$ by
\[ \tau(T) = \langle T \delta_e, \delta_e \rangle , \qquad T \in \rho_\Gamma^{\sigma}(\Gamma)' . \]
Note that if $V$ is a linear subspace of $\ell^2(\Gamma)$ invariant under $\rho_\Gamma^{\sigma}(\Gamma)$, then the projection $P_V$ onto the closure of $V$ is in $\rho_\Gamma^{\sigma}(\Gamma)'$. Similarly, if $T \in \rho_\Gamma^{\sigma}(\Gamma)'$, then the kernel $\mathcal{N}(T)$ and range $\mathcal{R}(T)$ are invariant under $\rho_\Gamma^{\sigma}(\Gamma)$.
The following properties are standard (see e.g.\ \cite{kleppner1962structure}) but we give their elementary proofs for the sake of being self-contained.

\begin{lemma}\label{lem:trace}
The following properties hold for $S,T \in \rho_\Gamma^{\sigma}(\Gamma)'$:
\begin{enumerate}
    \item[(i)] $\tau(ST) = \tau(TS)$.
    \item[(ii)] $\tau(I) = 1$.
    \item[(iii)] If $\tau(T^*T) = 0$, then $T=0$.
    \item[(iv)] $\tau(P_{\mathcal{N}(T)}) + \tau(P_{\mathcal{R}(T)}) = 1$.
\end{enumerate}
\end{lemma}

\begin{proof}
(i) For proving the claim, note that if $(T_i)_{i \in I}$ is any net in $\rho_\Gamma^{\sigma}(\Gamma)'$ converging to some $T \in \rho_\Gamma^{\sigma}(\Gamma)'$ in the weak operator topology, then $\tau(T_i) \to \tau(T)$ by definition. We prove the claim by first considering special instances of operators in $\rho_\Gamma^{\sigma}(\Gamma)'$ and then use a density argument.

Assume first that $S$ is a finite sum $S = \sum_\gamma c_\gamma \lambda_\Gamma^\sigma(\gamma)$ and that $T = \lambda_\Gamma^\sigma(\gamma')$ for some $\gamma' \in \Gamma$. Then
\[ \tau(ST) = \sum_\gamma c_\gamma \sigma(\gamma,\gamma') \langle \delta_{\gamma\gamma'}, \delta_e \rangle = c_{\gamma'^{-1}} \sigma(\gamma',\gamma'^{-1}) = \sum_\gamma c_\gamma \sigma(\gamma',\gamma) \langle \delta_{\gamma'\gamma}, \delta_e \rangle = \tau(TS). \]
By linearity, this extends to the case when $T$ is also such a finite sum. Next, let $T \in \rho_{\Gamma}^\sigma(\Gamma)'$ be arbitrary. Then $T$ can be expressed as the limit of a net $(T_i)_{i \in I}$ of the above finite sums in the weak operator topology. Since composition is separately continuous in the weak operator topology, it follows that $ST_i \to ST$ and $T_i S \to TS$, and so the first part of the proof gives
\[ \tau(ST) = \lim_i \tau(ST_i) = \lim_i \tau(T_iS) = \tau(TS) . \]
Finally, if $S \in \rho_\Gamma^{\sigma}(\Gamma)'$ is also arbitrary then expressing it as the limit of a net $(S_i)_{i \in I}$ of finite sums in the weak operator topology and using what we have already proved, we obtain
\[ \tau(ST) = \lim_i \tau(S_i T) = \lim_i \tau(TS_i) = \tau(TS) .  \]

(ii) $\tau(I) = \| \delta_e \|^2 = 1$.

(iii) Suppose that $\tau(T^*T) = 0$. Then for any $\gamma \in \Gamma$ we have that
\[ \| T \delta_\gamma \|^2 = \| T \lambda_\Gamma^\sigma(\gamma) \delta_e \|^2 = \| \lambda_\Gamma^\sigma(\gamma) T \delta_e \|^2 = \| T \delta_e \|^2 = \langle T^*T \delta_e, \delta_e \rangle = \tau(T^*T) = 0 . \]
Since the span of $\{ \delta_\gamma : \gamma \in \Gamma \}$ is dense in $\ell^2(\Gamma)$, it follows that $T = 0$.

(iv) Let $T = U|T|$ be the polar decomposition of $T$. Then $P_{\mathcal{R}(T)} = UU^*$ and $P_{\mathcal{R}(T)} = I - U^*U$. Hence, using (i) and (ii), it follows that
\[ \tau(P_{\mathcal{N}(T)}) + \tau(P_{\mathcal{R}(T)}) = \tau(I-U^*U) + \tau(UU^*) = \tau(I) = 1, \]
as desired.
\end{proof}

A countable discrete group $\Gamma$ is called \emph{amenable} if it admits a \emph{(left) Følner sequence}, that is, a sequence $(F_n)_{n \in \N}$ of finite subsets $F_n \subseteq \Gamma$ such that
\begin{equation}
    \lim_{n \to \infty} \frac{|F_n \, \triangle \, \gamma F_n |}{|F_n|} = 0 \qquad \text{for all $\gamma \in \Gamma$},
\end{equation}
where $A \, \triangle \, B$ denotes the symmetric difference of two sets $A$ and $B$. In particular, all abelian and nilpotent groups are amenable. An explicit Følner sequence in $\Gamma = \Z^d$ is given by the sets $F_n = \{ (k_1, \ldots, k_n) \in \Z^d : -n \leq k_i \leq n \}$ for $n \in \mathbb{N}$.

The following result shows that the existence of a nonzero linearly independent orbit in $\ell^2 (\Gamma)$ implies the existence of such an orbit in $\FS$.

\begin{proposition} \label{prop:l2zerodivisor}
  Suppose $\Gamma$ is an amenable group. If there exists nonzero $c \in \ell^2(\Gamma)$ whose orbit $\lambda_\Gamma^\sigma(\Gamma)c$ is linearly dependent, then there also exists nonzero $c' \in \FS$ such that $\lambda_\Gamma^\sigma(\Gamma) c'$ is linearly dependent.
\end{proposition}

\begin{proof}
Since $\lambda_\Gamma^\sigma(\Gamma) c$ is assumed to be linearly dependent, there exists nonzero $a \in \C \Gamma$ such that 
$ \sum_{\gamma \in \Gamma} a(\gamma) \lambda_{\Gamma}^{\sigma} (\gamma) c = 0$. Throughout the proof, we fix such a sequence $a$ and define the operator $C_a : \ell^2 (\Gamma) \to \ell^2 (\Gamma)$ by
\[
C_a= \sum_{\gamma \in \Gamma} a(\gamma) \lambda_{\Gamma}^{\sigma} (\gamma).
\]
Then for showing the claim it suffices to prove there exists nonzero $c' \in \FS$ such that $C_a c' = 0$. Note that $C_a$ commutes with $\rho_\Gamma^{\sigma}$, so that the kernel $\mathcal{N}(C_a)$ and range $\mathcal{R}(C_a)$ of $C_a$ are invariant under $\rho_\Gamma^{\sigma}$.

Set $K:= \supp(a)$ and let $(F_n)_{n \in \mathbb{N}}$ be a F\o lner sequence in $\Gamma$.
For fixed $n \in \mathbb{N}$, set 
$\interior_K (F_n) := \{ \gamma \in F_n : K \gamma \subseteq F_n \}$, and define the subspaces 
\[
V_n = \{ c \in \ell^2 (\Gamma) : \supp (c) \subseteq F_n \} \quad \text{and} \quad V'_n = \{ c \in \ell^2 (\Gamma) : \supp(c) \subseteq \interior_K (F_n) \}. 
\]
Note that $\supp(C_a c) \subseteq \supp(a)\supp(c)$ for all $c \in \ell^2(\Gamma)$. Since $K \interior_K(F_n) \subseteq F_n$, this implies that $C_a(V_n') \subseteq V_n$. Hence we consider the restriction of $C_a$ to $V'_n$ as a map $C^n_a := C_a|_{V_n'} : V_n' \to V_n$. Let $\mathcal{N}(C^n_a) \subseteq V'_n$ denote the kernel of $C_a^n$.

The proof will be split into two steps.
\\~\\
\textbf{Step 1.} In this step we show that the orthogonal projections $P_{\mathcal{N}(C_a)}$ and $P_{\mathcal{N}(C^n_a)}$ onto $\mathcal{N}(C_a)$  and $\mathcal{N}(C_a^n)$, respectively, satisfy the identity
\begin{align} \label{eq:dim_lim}
\tau(P_{\mathcal{N}(C_a)}) =  \lim_{n \to \infty} \frac{1}{|F_n|} \sum_{\gamma \in F_n} \| P_{\mathcal{N}(C^n_a)} \delta_{\gamma} \|^2.
\end{align}
For this, note that, given $n \in \mathbb{N}$, the kernel $\mathcal{N}(C_a^n) \subseteq V_n \cap \mathcal{N}(C_a) \subseteq \mathcal{N}(C_a)$, so $\| P_{\mathcal{N}(C_a^n)} \delta_\gamma \| \leq \| P_{\mathcal{N}(C_a)} \delta_\gamma \|$ for all $\gamma \in \Gamma$. In addition, since $\mathcal{N}(C_a)$ is invariant under $\rho_\Gamma^{\sigma}$, it follows that
\[ \| P_{\mathcal{N}(C_a)} \delta_\gamma \| = \| P_{\mathcal{N}(C_a)} \rho_{\Gamma}^\sigma(\gamma^{-1})\delta_e \| = \| \rho_{\Gamma}^\sigma(\gamma^{-1})P_{\mathcal{N}(C_a)} \delta_e \| = \| P_{\mathcal{N}(C_a)} \delta_e \| , \quad \gamma \in \Gamma. \]
Combining both observations yields
 \begin{align} \label{eq:dim_upper}
 \frac{1}{|F_n|} \sum_{\gamma \in F_n} \| P_{\mathcal{N}(C_a^n)} \delta_{\gamma} \|^2 \leq 
 \frac{1}{|F_n|} \sum_{\gamma \in F_n} \| P_{\mathcal{N}(C_a)} \delta_{\gamma} \|^2 = \tau( P_{\mathcal{N}(C_a)} ),
 \end{align}
Similarly, since the range $\mathcal{R}(C_a^n) \subseteq \mathcal{R}(C_a)$ and $\mathcal{R}(C_a)$ is $\rho_\Gamma^{\sigma}$-invariant, it follows that
  \begin{align} \label{eq:dim_upper2}
\frac{1}{|F_n|} \sum_{\gamma \in F_n} \| P_{\mathcal{R }(C_a^n)} \delta_{\gamma}\|^2 \leq \tau ( P_{\mathcal{R}(C_a)} ).
 \end{align}
Since both $\mathcal{N}(C_a^n)$ and $\mathcal{R}(C_a^n)$ are contained in $V_n$, it follows that $\| P_{\mathcal{N}(C_a^n)} \delta_\gamma \| = \| P_{\mathcal{R}(C_a^n)} \delta_\gamma \| = 0$ whenever $\gamma \notin F_n$. Hence
\begin{align*}
\frac{1}{|F_n|} \sum_{\gamma \in F_n} \| P_{\mathcal{N}(C_a^n)} \delta_{\gamma} \|^2 + \frac{1}{|F_n|} \sum_{\gamma \in F_n} \| P_{\mathcal{R }(C_a^n)} \delta_{\gamma} \|^2 &= \frac{\tr(P_{\mathcal{N}(C_a^n)}) + \tr(P_{\mathcal{R }(C_a^n)}) }{|F_n|} \\
&= \frac{\dim (\mathcal{N}(C_a^n)) + \dim (\mathcal{R}(C_a^n))}{|F_n|} \\
&= \frac{|\dim(V_n')|}{|F_n|} \\
&= \frac{|\interior_K (F_n)|}{|F_n|}. \numberthis \label{eq:folner0}
\end{align*}
We claim that 
\begin{align} \label{eq:folner2}
    \lim_{n \to \infty} \frac{|\interior_K (F_n)|}{|F_n|} = 1.
\end{align}
To see this, note that $\interior_K (F_n) = \bigcap_{k \in K \cup \{ e \} } k^{-1}F_n$, so that
\[ F_n \setminus \interior_K (F_n) = \bigcup_{k \in K \cup \{ e \} } F_n \cap k^{-1}F_n^c \subseteq \bigcup_{k \in K \cup \{e \}} F_n \triangle\, k^{-1}F_n .\]
Hence, as $n \to \infty$,
\[ \Big| 1 - \frac{|\interior_K (F_n)|}{|F_n|} \Big| = \frac{|F_n \setminus \interior_K (F_n)|}{|F_n|} \leq \sum_{k \in K \cup \{ e \}} \frac{|F_n \triangle\, k^{-1}F_n|}{|F_n|} \to 0 , \]
which proves the claim \eqref{eq:folner2}.

Combining \Cref{eq:folner0} and \Cref{eq:folner2} gives
\begin{align*}
\lim_{n \to \infty} \frac{1}{|F_n|} \sum_{\gamma \in F_n} \| P_{\mathcal{N}(C_a^n)} \delta_{\gamma} \|^2 + \frac{1}{|F_n|} \sum_{\gamma \in F_n} \| P_{\mathcal{R }(C_a^n)} \delta_{\gamma} \|^2
= 1.
\end{align*}
On the other hand it follows from \Cref{lem:trace} (iv) that
$\tau( P_{\mathcal{N}(C_a)} ) + \tau ( P_{\mathcal{R}(C_a)} ) = 1$. Thus, 
 \[
\lim_{n \to \infty} \frac{1}{|F_n|} \sum_{\gamma \in F_n} \| P_{\mathcal{N}(C_a^n)} \delta_{\gamma}\|^2 + \frac{1}{|F_n|} \sum_{\gamma \in F_n} \| P_{\mathcal{R }(C_a^n)} \delta_{\gamma}\|^2 =\tau ( P_{\mathcal{N}(C_a)} ) + \tau ( P_{\mathcal{R}(C_a)} ).
 \]
 Using \Cref{eq:dim_upper} and \Cref{eq:dim_upper2}, the claim \eqref{eq:dim_lim} follows.
\\~\\
\textbf{Step 2.} Since $C_a$ is nonzero it follows from \Cref{lem:trace} (iii) that $\tau(P_{C_a}) \neq 0$. By \Cref{eq:dim_lim}, there exists $n' \in \mathbb{N}$ such that $|F_{n'}|^{-1}\sum_{\gamma \in F_{n'}} \| P_{\mathcal{N}(C_a^{n'})} \delta_{\gamma} \|^2 > 0$, 
which implies that $\mathcal{N}(C^{n'}_a) \neq \{0\}$. Thus, there exists nonzero $c' \in V_{n'} \cap \mathcal{N}(C_a)$ which satisfies $c' \in \FS$ and $C_a c' = 0$. 
\end{proof}

The above argument is the same one as given in \cite[Theorem]{elek2003analytic}, but extended to twisted left regular representations of $\Gamma$.

\section{Zero divisors in twisted group rings}
This section is devoted to the study of zero divisors for twisted convolution. In particular, it will be shown that such nontrivial zero divisors do not exist for twisted convolutions on $\mathbb{Z}^d$, and more generally the class of locally indicable groups. In combination with the results obtained in \Cref{sec:linear}, this will allow us to provide the proofs of the main theorems in \Cref{sec:proofmain}.

Throughout this section, $\Gamma$ will denote a countable discrete group with identity element $e$ and $\sigma$ will denote a 2-cocycle on $\Gamma$ (cf. \Cref{eq:cocycle}). The \emph{$\sigma$-twisted convolution} of two sequences $a,b \in \FS$ is the sequence $a \ast_{\sigma} b$ in $\FS$ defined by
\begin{align*}
 (a *_\sigma b)(\gamma') 
 = \sum_{\gamma \in \Gamma} \sigma(\gamma,\gamma^{-1}\gamma')a(\gamma)b(\gamma^{-1}\gamma') , \qquad \, \quad \gamma' \in \Gamma.
 \end{align*}
Equipped with twisted convolution $\ast_{\sigma}$, the vector space $\FS$ is a complex algebra which will be denoted by $\mathbb{C} (\Gamma, \sigma)$ to emphasize the dependence on $\sigma$.
The algebra $\mathbb{C} (\Gamma, \sigma)$ is called a \emph{twisted group ring} or \emph{twisted group algebra} of $\Gamma$. As a vector space, $\C(\Gamma,\sigma)$ is spanned by the elements $\{ \delta_{\gamma} : \gamma \in \Gamma \}$, which are easily seen to satisfy the convolution relation
\begin{equation}
    \delta_\gamma *_\sigma \delta_{\gamma'} = \sigma(\gamma,\gamma') \delta_{\gamma\gamma'}, \quad \gamma, \gamma' \in \Gamma.
    \label{eq:delta}
\end{equation}
For $n \in \mathbb{N}$ and $a \in \FS$, the $n$-fold twisted convolution product of $a$ is denoted by $a^{\ast_{\sigma}(n)}$.
\subsection{Zero divisors}

An element $a \in \C(\Gamma,\sigma)$ is called a \emph{zero divisor} if there exists a nonzero $b \in \C (\Gamma,\sigma)$ such that $a *_\sigma b = 0$. The zero sequence is always a zero divisor, the so-called \emph{trivial zero divisor}. We will be concerned with when $\C(\Gamma,\sigma)$ has no zero divisors apart from the trivial one. 

First, recall  that $\Gamma$ is called \emph{torsion-free} if whenever $\gamma^n = e$ for some $\gamma \in \Gamma$ and positive integer $n$, then $\gamma = e$. The following extends a well-known observation for usual (nontwisted) group rings.

\begin{lemma}\label{lem:zero_divisor_torsion_free}
    If the twisted group ring $\C(\Gamma,\sigma)$ has no nontrivial zero divisors, then $\Gamma$ must be torsion-free.
\end{lemma}

\begin{proof}
    We give a proof by contraposition. Let $\gamma \in \Gamma$ be nontrivial and let $n \in \N$ be the least natural number such that $\gamma^n = e$. Note that by iterating \eqref{eq:delta}, it follows that
    \[ \delta_{\gamma}^{\ast_{\sigma}(n)} = \sigma(\gamma,\gamma) \sigma(\gamma,\gamma^2) \cdots \sigma(\gamma,\gamma^{n-1}) \delta_e , \]
    where $\delta_{\gamma}^{\ast_{\sigma}(n)}$ denotes the $n$-fold twisted convolution product of $\delta_{\gamma}$. Set 
    \[\alpha := \sigma(\gamma,\gamma) \sigma(\gamma,\gamma^2) \cdots \sigma(\gamma,\gamma^{n-1}) \quad \text{and} \quad a := \alpha^{-1/n} \delta_{\gamma}.\]
    Then $a^{\ast_{\sigma} (n)} = \delta_e$, and by expanding brackets one sees that
    \[ (\delta_e + a + a^{\ast_{\sigma} (2)} + \cdots + a^{\ast_{\sigma}(n-1)}) \ast_{\sigma} (a-\delta_e) = a^{\ast_\sigma(n)} - \delta_e = 0 . \]
    Note that for each $1 \leq k \leq n-1$, the support of the power $a^{\ast_{\sigma}(k)}$ equals $\{ \gamma^k \}$, and since each of these elements are distinct this implies that $\delta_e + a + a^{\ast_{\sigma}(2)} + \cdots + a^{\ast_{\sigma}(n-1)}$ is nonzero. This shows that $a-\delta_e$ is a nontrivial zero divisor in $\C(\Gamma,\sigma)$.
\end{proof}

For ordinary (nontwisted) group rings, the converse to \Cref{lem:zero_divisor_torsion_free} is the well-known \emph{zero divisor conjecture}, and is currently an open problem. More generally one can ask the following question.

\begin{question}\label{qu:twisted_zero_divisor}
Let $\Gamma$ be a torsion-free group and let $\sigma$ be a 2-cocycle on $\Gamma$. Does the twisted group ring $\C(\Gamma,\sigma)$ contain no nontrivial zero divisors?
\end{question}

The following example answers \Cref{qu:twisted_zero_divisor} affirmatively for $\Gamma = \Z^d$, $d \in \N$.

\begin{example}\label{ex:integers_zero_divisor}
Let $\Gamma = \Z^d$ and let $\sigma$ be any 2-cocycle on $\Z^d$. We give a direct proof that the twisted group ring $\C(\Z^d,\sigma)$ does not contain any nontrivial zero divisors. Let $e_i$ denote the $i$th basis vector of $\Z^d$, $1 \leq i \leq d$, and set $u_i = \delta_{e_i}$. Note that elements $a \in \C(\Z^d,\sigma)$ can be written as multivariate polynomials
\begin{align*}
    a = \sum_{i_1, \ldots, i_d \in \Z} a_{i_1, \ldots, i_d} u_1^{\ast_{\sigma}(i_1)} \cdots u_d^{\ast_{\sigma}(i_d)}
\end{align*} 
where the coefficients $a_{i_1,\ldots,i_d} \in \C$ are zero for all but finitely many indices. The multiplication is not commutative like in an ordinary polynomial ring, but instead governed by the basic noncommutative relations $u_i u_j = z_{i,j} u_j u_i$ where $z_{i,j} = \sigma(e_i,e_j)\overline{\sigma(e_j,e_i)}$.

Let us call $a$ non-negative if $a_{i_1,\ldots,i_d} = 0$ whenever $i_k < 0$ for some $1 \leq k \leq d$. We define the degree of a non-negative, nonzero element $a$ to be the maximum of the numbers $i_1 + \cdots i_d$ where $a_{i_1,\ldots,i_d} \neq 0$, and set $\deg(0) = -1$. One can then verify as with usual polynomial multiplication that $\deg(a *_\sigma b) = \deg(a) + \deg(b)$ for non-negative elements $a,b \in \C(\Z^d,\sigma)$.

Letting now $a,b \in \C(\Z^d,\sigma)$ be nonzero, so that $\deg(a) \geq 0$ and $\deg(b) \geq 0$, we wish to prove that $a *_\sigma b \neq 0$. Note that by multiplying with a high enough power of $u_1 \cdots u_d$ we may assume that $a$ and $b$ are non-negative. But then $\deg(a *_\sigma b) = \deg(a) + \deg(b) \geq 0$ which implies that $a *_\sigma b \neq 0$.
\end{example}

\subsection{Locally indicable groups}

Our next result (\Cref{prop:zerodivisor_indicable}) 
shows that \Cref{qu:twisted_zero_divisor} is affirmative for locally indicable groups. This extends a classical result \cite{higman1940units} to the twisted setting. We start by introducing the relevant notions and terminology. Readers only interested in the case $\Gamma = \mathbb{Z}^d$ and \Cref{thm:linnell} may skip this subsection.

A \emph{degree map} on $\Gamma$ is a surjective group homomorphism $\phi \colon \Gamma \to \Z$. Given $\gamma \in \Gamma$, we refer to $\phi(\gamma)$ as the \emph{degree} of $\gamma$ (relative to $\phi$). An element $a \in \C(\Gamma,\sigma)$ is called \emph{homogeneous of degree $k$} if $\phi(\gamma) = k$ for all $\gamma \in \supp(a)$.
Note that if $a$ is homogeneous of degree $k$ and $b$ is homogeneous of degree $l$, then $a *_\sigma b$ is homogeneous of degree $k+l$: Indeed, if $\gamma \in \supp(a *_\sigma b)$ then $\gamma = \gamma_1 \gamma_2$ for some $\gamma_1 \in \supp(a)$ and $\gamma_2 \in \supp(b)$, which means that $\phi(\gamma) = \phi(\gamma_1)+\phi(\gamma_2) = k+l$. Every element of $\C(\Gamma,\sigma)$ can be written as a sum of homogeneous elements. If $a$ is a sum of homogeneous elements $a_1, \ldots, a_m$ of degrees $k_1, \ldots, k_m$ and $b$ is a sum of homogeneous elements $b_1, \ldots, b_n$ of degrees $l_1, \ldots, l_n$, then $a *_\sigma b$ is the sum of the homogeneous elements $a_i *_\sigma b_j$ of degrees $k_i + l_j$, $1 \leq i \leq m$ and $1 \leq j \leq n$.

The group $\Gamma$ is said to be \emph{locally indicable} if every nontrivial, finitely generated subgroup of $\Gamma$ admits a degree map.

\begin{proposition} \label{prop:zerodivisor_indicable}
If $\Gamma$ is locally indicable, then $\C(\Gamma,\sigma)$ contains no nontrivial zero divisors.
\end{proposition}

\begin{proof}
We will prove the following statement by induction on $n$: For every nonzero $a,b \in \C(\Gamma,\sigma)$ with $|\supp(a)|+|\supp(b)| = n$, the twisted convolution $a *_\sigma b$ is nonzero. For the base case $n=2$ we have that $|\supp(a)|=|\supp(b)| = 1$, say $\supp(a) = \{ \gamma_1 \}$ and $\supp(b) = \{ \gamma_2 \}$, so that $a *_\sigma b = a(\gamma_1)b(\gamma_2)\sigma(\gamma_1,\gamma_2) \delta_{\gamma_1 \gamma_2}$. Then $(a *_\sigma b)(\gamma_1 \gamma_2) = a(\gamma_1) b(\gamma_2) \sigma(\gamma_1,\gamma_2) \neq 0$, which means that $\gamma_1 \gamma_2 \in \supp(a *_\sigma b)$. Hence, the base case is proved.

For the induction step, let $n \in \N$, $n \geq 3$, and assume that the statement holds for all $k < n$. Let $a,b \in \C(\Gamma,\sigma)$ be nonzero with $|\supp(a)|+|\supp(b)|=n$. For any $\gamma_1 \in \supp(a)$ and $\gamma_2 \in \supp(b)$, consider $a' = \delta_{{\gamma_1}^{-1}} *_\sigma a$ and $b' = b *_\sigma \delta_{{\gamma_2}^{-1}}$. Note that $a *_\sigma b = 0$ if and only if $a' *_\sigma b' = 0$ and that $e \in \supp(a') \cap \supp(b')$. Hence, by replacing $a$ with $a'$ and $b$ with $b'$ respectively, it may be assumed that $e$ is contained in the support of both $a$ and $b$.

Let $\Gamma_0$ be the subgroup of $\Gamma$ generated by $\supp(a) \cup \supp(b)$. Since $n \geq 3$, $\Gamma_0$ is nontrivial, so there exists a degree map $\phi \colon \Gamma_0 \to \Z$. We claim that $a$ and $b$ cannot both be homogeneous (relative to $\phi$). Indeed, if both $a$ and $b$ where homogeneous, then they would be homogeneous of degree $0$ since $e \in \supp(a) \cap \supp(b)$, which would imply that $\phi(\gamma) = 0$ for all $\gamma \in \supp(a) \cup \supp(b)$
, hence $\phi \equiv 0$. This contradicts the surjectivity of $\phi$.

Let $k = \min \{ \phi(\gamma) : \gamma \in \supp(a) \}$ and $l = \min \{ \phi(\gamma) : \gamma \in \supp(b) \}$. Let
\[ a' = \sum_{\gamma \in \phi^{-1}(k)} a(\gamma) \delta_\gamma \quad \text{and} \quad b' = \sum_{\gamma \in \phi^{-1}(l)} b(\gamma) \delta_{\gamma}.
\]
In other words, $a'$ is the homogeneous element of least degree in the expansion of $a$ into homogeneous elements, and analogously for $b'$ with respect to $b$. By the previous paragraph either $\supp(a') \subsetneq \supp(a)$ or $\supp(b') \subsetneq \supp(b)$. In either case $|\supp(a')|+|\supp(b')| < n$, so by the induction hypothesis $a' *_\sigma b' \neq 0$. Expanding $a *_\sigma b$ into a sum of homogeneous elements, we obtain
\[ a *_\sigma b = a' *_\sigma b' + R,  \]
where $a' *_\sigma b'$ is homogeneous of degree $k+l$ and $R$ is a sum of homogeneous elements of degrees strictly bigger than $k+l$. Consequently, the supports of $a' *_\sigma b'$ and $R$ are disjoint, so $a' *_\sigma b' \neq 0$ implies that $a *_\sigma b \neq 0$. This finishes the proof.
\end{proof}

The proof of \Cref{prop:zerodivisor_indicable} follows the proof for ordinary (nontwisted) group rings in \cite{higman1940units}.

Lastly, we provide a simple argument showing that nilpotent groups are locally indicable. Recall that $\Gamma$ is nilpotent if 
the upper central series defined recursively by $Z_0 := \{e \}$ and
\[ Z_{n+1} := \{ \gamma \in \Gamma : [\gamma,\gamma'] \in Z_n \; \text{for all $\gamma' \in \Gamma$} \}, \quad n \in \N , \]
terminates after a finite number of steps, that is, $Z_n = \Gamma$ for some $n \in \N$. The smallest such $n$ is called the nilpotency class of $\Gamma$. Note that $Z_1 = Z(\Gamma)$, the center of $\Gamma$.

\begin{lemma} \label{lem:nilpotent_indicable}
Every torsion-free nilpotent group $\Gamma$ is locally indicable.
\end{lemma}

\begin{proof}
We prove the lemma by induction on the nilpotency class of $\Gamma$. If $\Gamma$ has nilpotency class $1$, then it is abelian and torsion-free, hence every finitely generated subgroup of $\Gamma$ is free abelian by the classification of finitely generated abelian groups. For a free abelian group, the projection onto the subgroup generated by any one of its generators is a degree map. This shows that $\Gamma$ is locally indicable when it has nilpotency class $1$.

Next, suppose the result holds for all nilpotent, torsion-free groups of class strictly smaller than $n$ and let $\Gamma$ be torsion-free and nilpotent of class $n$. Firstly, note that the center $Z(\Gamma)$ of $\Gamma$ is abelian and torsion-free, hence locally indicable by the base case of the induction. Secondly, the quotient $\Gamma/Z(\Gamma)$ is nilpotent of class strictly smaller than $n$, and also necessarily torsion-free (see, e.g.,\ \cite[Corollary 2.22]{Clement2017}). Thus $\Gamma/Z(\Gamma)$ is locally indicable by the induction hypothesis. Now that both $Z(\Gamma)$ and $\Gamma/Z(\Gamma)$ are locally indicable,  it follows from \cite[p.\ 246, Lemma]{higman1940units} that $\Gamma$ is locally indicable. This finishes the proof.
\end{proof}

\section{Proof of main theorems} \label{sec:proofmain}

The following general theorem shows that the problem of linearly independent orbits of square-integrable representations over discrete subgroups 
can be reduced to the existence of zero divisors in twisted group rings.

 \begin{theorem} \label{thm:intro3}
    Let $G$ be a second-countable, locally compact group  and let $(\pi, \Hpi)$ be a $\sigma$-projective unitary representation of $G$ admitting admissible vectors. If $\Gamma$ is a discrete amenable subgroup of $G$ such that the twisted group ring $\C (\Gamma,\sigma)$ contains no nontrivial zero divisors, then the coherent system
    \[ \pi(\Gamma)g = \{ \pi(\gamma)g : \gamma \in \Gamma \} \]
    is linearly independent for any nonzero vector $g \in \Hpi$.
\end{theorem}

\begin{proof}
Arguing by contraposition, suppose that the coherent system $\pi(\Gamma)g$ is linearly dependent. By \Cref{lem:lineardependG}, this implies that there exists nonzero $F \in L^2(G)$ such that $\lambda_G^\sigma(\Gamma)F$ is linearly dependent in $L^2(G)$. By \Cref{lem:reductionGamma} this implies again that there exists a nonzero $c \in \ell^2(\Gamma)$ such that $\lambda_\Gamma^\sigma(\Gamma)c$ is linearly dependent in $\ell^2(\Gamma)$. Finally, since $\Gamma$ is assumed amenable, we can apply \Cref{prop:l2zerodivisor} to conclude that $\C(\Gamma,\sigma)$ contains a nontrivial zero divisor.
\end{proof}

\Cref{thm:linnell} and \Cref{thm:intro2} are now a simple consequence of \Cref{thm:intro3}:

\begin{proof}[Proof of \Cref{thm:linnell}]
Let $\Gamma$ be a discrete subgroup of $\R^{2d}$. By the orthogonality relations of the short-time Fourier transform \cite[Chapter 3]{grochenig2001foundations}, any unit vector $g \in L^2 (\mathbb{R}^d)$ is admissible. Hence, by \Cref{thm:intro3}, it suffices to show that $\C(\Gamma,\sigma)$, where $\sigma$ denotes the 2-cocycle from Equation \eqref{eq:time-freq-cocycle}, contains no nontrivial zero divisors. However, a discrete subgroup of $\R^{2d}$ is isomorphic to $\Z^k$ for for some $0 \leq k \leq 2d$, hence the twisted group ring $\C(\Gamma,\sigma)$ is isomorphic to $\C(\Z^k,\sigma')$ for a 2-cocycle $\sigma'$ on $\Z^k$. The nonexistence of nontrivial zero divisors in the latter was established in \Cref{ex:integers_zero_divisor}.
\end{proof}

\begin{proof}[Proof of \Cref{thm:intro2}]
Since $G$ is a unimodular group, it follows by the orthogonality relations \cite{duflo1976on, carey1976square} that any nonzero vector $g \in \Hpi$ is (a multiple of) an admissible vector.
A discrete subgroup $\Gamma$ of a connected, simply connected nilpotent Lie group $G$ is  torsion-free and nilpotent, cf. \cite[Chapter 2]{raghunathan1972discrete}. Thus, $\Gamma$ is locally indicable by Lemma \ref{lem:nilpotent_indicable}. The claim follows therefore directly from \Cref{thm:intro3}.
\end{proof}

\section*{Acknowledgements}
U.E.\ gratefully acknowledges support from the The Research Council of Norway through project 314048. For J.~v.~V., this research was funded in whole or in part by the Austrian
Science Fund (FWF): 10.55776/J4555 and 10.55776/PAT2545623. For open access purposes, the author has applied a CC BY public copyright license to any author-accepted manuscript version arising from this submission.

\bibliographystyle{abbrv}
\bibliography{bib}

\end{document}